\documentclass[10pt,a4paper]{article}

\usepackage{latexsym,amsfonts,amsmath,amssymb,amsthm,mathrsfs,color}

\newtheorem{theorem}{Theorem}
\newtheorem{lemma}{Lemma}

\newtheorem{remark}{Remark}
\newtheorem{proposition}{Proposition}
\newtheorem{definition}{Definition}
\newtheorem{corollary}{Corollary}

\newcommand{\rd}{\, \mathrm{d}}
\newcommand{\bszero}{\boldsymbol{0}}
\newcommand{\bsc}{\boldsymbol{c}}
\newcommand{\bsk}{\boldsymbol{k}}
\newcommand{\bsl}{\boldsymbol{l}}
\newcommand{\bsr}{\boldsymbol{r}}
\newcommand{\bsx}{\boldsymbol{x}}
\newcommand{\bsy}{\boldsymbol{y}}
\newcommand{\bsw}{\boldsymbol{w}}
\newcommand{\bsalpha}{\boldsymbol{\alpha}}

\newcommand{\CC}{\mathbb{C}}
\newcommand{\FF}{\mathbb{F}}
\newcommand{\NN}{\mathbb{N}}
\newcommand{\RR}{\mathbb{R}}

\newcommand{\Ecal}{\mathcal{E}}
\newcommand{\Hcal}{\mathcal{H}}
\newcommand{\Kcal}{\mathcal{K}}
\newcommand{\per}{\mathrm{per}}
\newcommand{\tr}{\mathrm{tr}}
\newcommand{\wal}{\mathrm{wal}}
\newcommand{\wor}{\mathrm{wor}}

\allowdisplaybreaks

\begin{document}

\title{An explicit construction of optimal order quasi-Monte Carlo rules for smooth integrands\thanks{The work of T.~G. is supported by JSPS Grant-in-Aid for Young Scientists No.15K20964.
The work of K.~S. and T.~Y. is supported by Australian Research Council's Discovery Projects funding scheme (project number DP150101770).}}

\author{Takashi Goda\thanks{Graduate School of Engineering, The University of Tokyo, 7-3-1 Hongo, Bunkyo-ku, Tokyo 113-8656, Japan (\tt{goda@frcer.t.u-tokyo.ac.jp})}, Kosuke Suzuki\thanks{School of Mathematics and Statistics, The University of New South Wales, Sydney 2052, Australia. ({\tt kosuke.suzuki1@unsw.edu.au})}, Takehito Yoshiki\thanks{School of Mathematics and Statistics, The University of New South Wales, Sydney 2052, Australia. ({\tt takehito.yoshiki1@unsw.edu.au})}}

\date{\today}

\maketitle

\begin{abstract}
In a recent paper by the authors, it is shown that there exists a quasi-Monte Carlo (QMC) rule which achieves the best possible rate of convergence for numerical integration in a reproducing kernel Hilbert space consisting of smooth functions.
In this paper we provide an explicit construction of such an optimal order QMC rule.
Our approach is to exploit both the decay and the sparsity of the Walsh coefficients of the reproducing kernel simultaneously.
This can be done by applying digit interlacing composition due to Dick to digital nets with large minimum Hamming and Niederreiter-Rosenbloom-Tsfasman metrics due to Chen and Skriganov.
To our best knowledge, our construction gives the first QMC rule which achieves the best possible convergence in this function space.
\end{abstract}

%%%%%%%%%%%%%%%%%%%%%%%%%%%%%%%%%%%%%%%%%%%%%%%%%%%%%%%%%%%%%%%%%%%%%%%%%%%%%%%%%%%%%%%%%%%%%%%%%%%
%%%%%%%%%%%%%%%%%%%%%%%%%%%%%%%%%%%%%%%%%%%%%%%%%%%%%%%%%%%%%%%%%%%%%%%%%%%%%%%%%%%%%%%%%%%%%%%%%%%
\section{Introduction and the main result}\label{sec:intro}
In this paper we study numerical integration of smooth functions defined on the $s$-dimensional unit cube.
For an integrable function $f:[0,1)^s\to \RR$, we denote the true integral of $f$ by 
\begin{align*}
I(f) = \int_{[0,1)^s}f(\bsx) \rd \bsx.
\end{align*}
For an $N$ element point set $P=\{\bsx_0,\ldots,\bsx_{N-1}\}\subset [0,1)^s$ and an arbitrary real tuple $\bsw=\{w_0,\ldots,w_{N-1}\}$,
we consider a quadrature rule of the form
\begin{align*}
I(f; P,\bsw) = \sum_{n=0}^{N-1}w_nf(\bsx_n) ,
\end{align*}
as an approximation of $I(f)$.
In particular, we are interested in the case where $w_0=\cdots =w_{N-1}=1/N$, i.e., an approximation by an equal-weight quadrature rule where the weights sum up to 1.
This type of quadrature rule is called a \emph{quasi-Monte Carlo (QMC) rule} and has been extensively studied in the literature, see for instance \cite{DPbook,Niebook,SJbook}.
Since a QMC rule depends only on the choice of a point set $P$, we simply write
\begin{align*}
I(f; P) = \frac{1}{|P|}\sum_{\bsx\in P}f(\bsx) ,
\end{align*}
to denote a QMC rule using a point set $P$, where points are counted according to their multiplicity.

We measure the quality of a quadrature rule in terms of the so-called \emph{worst-case error}.
Let $V$ be a function space with norm $\lVert \cdot \rVert_V$.
The worst-case error of a quadrature rule in $V$ is defined as the supremum of the approximation error in the unit ball of $V$, i.e.,
\begin{align*}
e^{\wor}(V;P,\bsw) := \sup_{\substack{f\in V\\ \lVert f \rVert_V\le 1}}\left| I(f; P,\bsw)-I(f)\right| .
\end{align*}
In case of a QMC rule, we simply write
\begin{align*}
e^{\wor}(V;P) := \sup_{\substack{f\in V\\ \lVert f \rVert_V\le 1}}\left| I(f; P)-I(f)\right| .
\end{align*}
As an important example of a normed space, a reproducing kernel Hilbert space of Sobolev type consisting of functions with relatively low smoothness has been often considered in the literature, see for instance \cite[Section~2]{DPbook} and \cite{Hick98}.
Such a function space is not only connected to geometric discrepancies of point sets, but also considered relevant to financial applications \cite[Appendix~A]{NWbook}.

Function spaces with high smoothness have received considerable attention in recent applications in the area of partial differential equations with random coefficients, see for instance \cite{DKLNS14,KSS12}.
In fact, such applications are in need of using quadrature rules which can exploit the smoothness of functions and achieve high order convergence.
Dick and his collaborators \cite{BD09,BDP11,Dick07,Dick08} have developed an important class of QMC rules named \emph{higher order digital nets} achieving almost optimal convergence of order $N^{-\alpha}(\log N)^{c(s,\alpha)}$ for functions with smoothness $\alpha\in \NN$, $\alpha \geq 2$, for some $c(s,\alpha)>0$.
Although this order of convergence is the best possible up to some power of a $\log N$ factor, it has been unknown until recently whether the exponent $c(s,\alpha)$ can be improved to optimal.
As far as the authors know, there are only two papers addressing this issue \cite{GSYexist,HMOT15}.
In \cite{HMOT15} Hinrichs et al.\ considered periodic Sobolev spaces and periodic Nikol'skij-Besov spaces with (real-valued) dominating mixed smoothness up to 2, and obtained $c(s,\alpha)=(s-1)/2$ for order 2 digital nets, which is best possible.
In \cite{GSYexist} the authors of this paper considered a reproducing kernel Hilbert space of Sobolev type consisting of non-periodic functions with smoothness $\alpha\in \NN$, $\alpha \geq 2$, and proved the existence of a digitally shifted order $\beta$ digital nets achieving $c(s,\alpha)=(s-1)/2$ when $\beta \geq 2\alpha$.
Although the resulting value of $c(s,\alpha)$ is best possible, a random element that stems from digital shift was involved in the result so that the construction is not explicit.

In this study, as a continuation of the paper \cite{GSYexist}, we provide an explicit construction of higher order digital nets which achieve the best possible order of convergence without requiring a random element.
Our main idea is to exploit both the decay and the sparsity of the Walsh coefficients of the reproducing kernel simultaneously.
The decay of the Walsh coefficients can be exploited by digit interlacing composition due to Dick \cite{Dick07,Dick08}, whereas the sparsity of the Walsh coefficients can be exploited by digital nets with large minimum Hamming and Niederreiter-Rosenbloom-Tsfasman (NRT) metrics due to Chen and Skriganov \cite{CS02,Skr06}.
Hence our construction is simply given by applying Dick's digit interlacing composition to Chen-Skriganov's digital nets, which shall be discussed in detail in Section~\ref{sec:constr}.
Historically, Chen-Skriganov's digital nets are the first explicit construction of QMC point sets for any dimension $s$ with the best possible $L_p$ discrepancy for each $1<p<\infty$, and their approach is indeed to exploit the sparsity of the Walsh coefficients of the characteristic function of axes-parallel rectangles anchored in zero.
In this paper, instead of the characteristic function, we shall deal with the Walsh coefficients of Bernoulli polynomials as studied in \cite{BD09,Dick09,SYwalsh}, see also \cite{Ywalsh} from which a similar result can be derived.

We now state the main result of this paper.
Here $\Hcal_{\alpha,s}$ denotes an $s$-variate Sobolev space with smoothness $\alpha\in \NN$, $\alpha\geq 2$. 
(We shall give the precise definition of $\Hcal_{\alpha,s}$ later in Subsection~\ref{subsec:wSobolev}.)

\begin{theorem}\label{thm:main}
Let $s,\alpha\in \NN$, $\alpha \geq 2$.
Let $\beta,g\in \NN$ with $\beta \geq 2\alpha$, $g\geq 2\alpha s$ and $g\geq \lfloor s(\beta -1)/2\rfloor$, and let $b\geq \beta gs$ be a prime.
Then for every $w\in \NN$, we can explicitly construct a point set $P$ of cardinality $N=b^{gw}$ such that
\begin{align*}
e^{\wor}(\Hcal_{\alpha,s};P) \leq C_{\alpha,\beta,b,s}\frac{(\log N)^{(s-1)/2}}{N^{\alpha}},
\end{align*}
where $C_{\alpha,\beta,b,s}$ is positive and independent of $w$.
\end{theorem}

This theorem implies that the convergence of order $N^{-\alpha}(\log N)^{(s-1)/2}$ can be achieved in this function space. As already mentioned, this order is actually best possible, as can be seen from the lower bound on $e^{\wor}(\Hcal_{\alpha,s};P,\bsw)$ for any $P$ and $\bsw$, see \cite[Proposition~1]{GSYexist}.
The same order of convergence can be also achieved in a similar function space by using the Frolov lattice rule in conjunction with periodization strategy \cite{Frolov76,Ullrich14,UU15}.
The Frolov lattice rule is an equal-weight quadrature rule, although the weights do not sum up to 1 in general.
Therefore, the Frolov lattice rule is not a QMC rule.
To our best knowledge, our presented construction gives the first QMC rule which achieves the best possible convergence in $\Hcal_{\alpha,s}$.

The remainder of this paper is organized as follows.
In Section~\ref{sec:pre}, we shall introduce the necessary background and notation, including digital nets, Sobolev spaces, and Walsh functions.
In Section~\ref{sec:constr}, after reviewing Chen-Skriganov's digital nets and Dick's digit interlacing composition, we shall give an explicit construction of point sets which achieve the best possible convergence in $\Hcal_{\alpha,s}$.
Finally in Section~\ref{sec:proof}, we shall give the proof of the main result.

%%%%%%%%%%%%%%%%%%%%%%%%%%%%%%%%%%%%%%%%%%%%%%%%%%%%%%%%%%%%%%%%%%%%%%%%%%%%%%%%%%%%%%%%%%%%%%%%%%%
%%%%%%%%%%%%%%%%%%%%%%%%%%%%%%%%%%%%%%%%%%%%%%%%%%%%%%%%%%%%%%%%%%%%%%%%%%%%%%%%%%%%%%%%%%%%%%%%%%%
\section{Preliminaries}\label{sec:pre}
Throughout this paper we shall use the following notation.
Let $\NN$ be the set of positive integers and $\NN_0=\NN\cup \{0\}$.
Let $\CC$ be the set of all complex numbers.
For a prime $b$, let $\FF_b$ be the finite field with $b$ elements, which is identified with the set $\{0,1,\ldots,b-1\}$ equipped with addition and multiplication modulo $b$.
The operators $\oplus$ and $\ominus$ denote digitwise addition and subtraction modulo $b$, respectively, that is, for $k=\sum_{i=1}^{\infty}\kappa_ib^{i-1}\in \NN_0$ and $k'=\sum_{i=1}^{\infty}\kappa'_ib^{i-1}\in \NN_0$ with $\kappa_i,\kappa'_i\in \FF_b$, which are actually finite expansions, we define 
\begin{align*}
k\oplus k' := \sum_{i=1}^{\infty}\lambda_ib^{i-1}\quad \text{and}\quad k\ominus k' := \sum_{i=1}^{\infty}\lambda'_ib^{i-1} ,
\end{align*}
where $\lambda_i=\kappa_i+\kappa'_i \pmod b$ and $\lambda'_i=\kappa_i-\kappa'_i \pmod b$.
In case of vectors in $\NN_0^s$, the operators $\oplus$ and $\ominus$ are applied componentwise.

%%%%%%%%%%%%%%%%%%%%%%%%%%%%%%%%%%%%%%%%%%%%%%%%%%%%%%%%%%%%%%%%%%%%%%%%%%%%%%%%%%%%%%%%%%%%%%%%%%%
\subsection{Digital nets}
Here we introduce the definition of digital nets over $\FF_b$ due to Niederreiter \cite{Niebook}.
\begin{definition}
For a prime $b$ and $s,m,n\in \NN$, let $C_1,\ldots,C_s\in \FF_b^{n\times m}$. For each integer $0\leq h< b^m$, denote the $b$-adic expansion of $h$ by $h=\eta_0+\eta_1b+\cdots +\eta_{m-1}b^{m-1}$ with $\eta_0,\eta_1,\ldots,\eta_{m-1}\in \FF_b$. For $1\leq j\leq s$, consider 
\begin{align*}
x_{h,j}=\frac{\xi_{1,h,j}}{b}+\frac{\xi_{2,h,j}}{b^2}+\cdots +\frac{\xi_{n,h,j}}{b^n}\in [0,1) ,
\end{align*}
where $\xi_{1,h,j},\xi_{2,h,j},\ldots,\xi_{n,h,j}\in \FF_b$ are given by
\begin{align*}
(\xi_{1,h,j},\xi_{2,h,j},\ldots,\xi_{n,h,j})^{\top} = C_j\cdot (\eta_0,\eta_1,\ldots,\eta_{m-1})^{\top}.
\end{align*}
Then the set $P=\{\bsx_0,\bsx_1,\ldots,\bsx_{b^m-1}\}\subset [0,1)^s$ with $\bsx_h=(x_{h,1},\ldots,x_{h,s})$ is called a \emph{digital net over $\FF_b$} with generating matrices $C_1,\ldots,C_s$.
\end{definition}

The concept of dual nets shall play a crucial role in our subsequent analysis.
We give two different notions of dual nets below.
\begin{definition}
For a prime $b$ and $s,m,n\in \NN$, let $P$ be a digital net over $\FF_b$ with generating matrices $C_1,\ldots,C_s\in \FF_b^{n\times m}$. We define
\begin{align*}
P^{\perp} := \left\{ \bsk\in \{0,1,\ldots,b^n-1\}^s \colon C_1^{\top}\tr_n(k_1)\oplus \cdots \oplus C_s^{\top}\tr_n(k_s)=\bszero \in \FF_b^m \right\}
\end{align*}
and
\begin{align*}
P_{\infty}^{\perp} := \left\{ \bsk\in \NN_0^s\colon C_1^{\top}\tr_n(k_1)\oplus \cdots \oplus C_s^{\top}\tr_n(k_s)=\bszero \in \FF_b^m \right\},
\end{align*}
where $\bsk=(k_1,\ldots,k_s)$ and we denote $\tr_n(k)=(\kappa_0,\ldots,\kappa_{n-1})^{\top}\in \FF_b^{n}$ for $k\in \NN_0$ with $b$-adic expansion $k=\kappa_0+\kappa_1 b+\cdots$, which is actually a finite expansion.
\end{definition}
\noindent Here we note the difference between $P^{\perp}$ and $P_{\infty}^{\perp}$.
Obviously we have $P^{\perp}\subseteq \{0,1,\ldots,b^n-1\}^s$, $P_{\infty}^{\perp}\subseteq \NN_0^s$, $P^{\perp}\subset P_{\infty}^{\perp}$ and $P_{\infty}^{\perp}\setminus P^{\perp}\subseteq \NN_0^s\setminus \{0,1,\ldots,b^n-1\}^s$ for any digital net $P$ with finite $n$.
The case $n=\infty$, where the difference between $P^{\perp}$ and $P_{\infty}^{\perp}$ vanishes, has been discussed in \cite{GSYtent2}, although we shall only consider the case where $n$ is finite in this paper.

We introduce two metric functions on $\NN_0$: the Hamming metric and the Dick metric. Note that the Dick metric introduced in \cite{Dick07,Dick08} is a generalization of the NRT metric introduced in \cite{Nied86,RT97}, which itself is a generalization of the Hamming metric.
\begin{definition}
For $k\in \NN$, we denote its $b$-adic expansion by $k=\kappa_1 b^{a_1-1}+\cdots +\kappa_v b^{a_v-1}$ with $\kappa_1,\ldots,\kappa_v\in \{1,\ldots,b-1\}$ and $a_1>\cdots >a_v>0$. For $k=0$, we assume that $v=0$ and $a_0=0$.
\begin{enumerate}
\item The Hamming metric $\varkappa:\NN_0\to \RR$ is defined as the number of non-zero digits in the $b$-adic expansion, i.e., $\varkappa(k):=v$.
\item Let $\alpha \in \NN$. The Dick metric $\mu_{\alpha}:\NN_0\to \RR$ is defined by
\begin{align*}
\mu_{\alpha}(k) := \sum_{i=1}^{\min(v,\alpha)}a_i.
\end{align*}
In particular, the metric $\mu_{1}$ is called the NRT metric.
\end{enumerate}
In case of a vector $\bsk=(k_1,\ldots,k_s)\in \NN_0^s$, we define
\begin{align*}
\varphi(\bsk) := \sum_{j=1}^{s}\varphi(k_j),
\end{align*}
for $\varphi\in \{\varkappa,\mu_{\alpha}\}$.
\end{definition}

For a digital net $P$, we define its minimum metric by
\begin{align*}
\varphi(P)  := \min_{\substack{\bsk,\bsl\in P^{\perp}\\ \bsk\ne \bsl}}\varphi(\bsk\ominus \bsl) = \min_{\bsk\in P^{\perp}\setminus \{\bszero\}}\varphi(\bsk),
\end{align*}
for $\varphi\in \{\varkappa,\mu_{\alpha}\}$.
In the above, the latter equality stems from the fact that $P^{\perp}$ is a subgroup of $\{0,1,\ldots,b^n-1\}^s$ with the group operation $\oplus$.
Roughly speaking, the minimum metric of $P$ measures how uniformly $P$ is distributed in $[0,1)^s$.
For instance, explicit constructions due to Sobol' \cite{Sobol67}, Faure \cite{Faure82}, Niederreiter \cite{Nied88}, and Niederreiter and Xing \cite{NXbook} provide digital nets with large minimum NRT metric and the star discrepancy of such digital nets is known to decay with order $(\log N)^{s-1}/N$.
We refer to \cite[Section~8]{DPbook} for more information on these constructions.

%%%%%%%%%%%%%%%%%%%%%%%%%%%%%%%%%%%%%%%%%%%%%%%%%%%%%%%%%%%%%%%%%%%%%%%%%%%%%%%%%%%%%%%%%%%%%%%%%%%
\subsection{Sobolev spaces}\label{subsec:wSobolev}
Here we introduce the function space which we deal with in this paper.
First let us consider the univariate case.
For a given $\alpha\in \NN$, $\alpha\geq 2$, the Sobolev space with smoothness $\alpha$ which we consider is given by
\begin{align*}
 \Hcal_{\alpha} & := \Big\{f \colon [0,1)\to \RR \mid \\
 & \qquad f^{(r)} \colon \text{absolutely continuous for $r=0,\ldots,\alpha-1$}, f^{(\alpha)}\in L^2[0,1)\Big\},
\end{align*}
where $f^{(r)}$ denotes the $r$-th derivative of $f$.
As in \cite[Section~10.2]{Wbook} this space is indeed a reproducing kernel Hilbert space with the reproducing kernel $\Kcal_{\alpha}\colon [0,1)\times [0,1)\to \RR$ and the inner product $\langle \cdot, \cdot \rangle_{\alpha}$ given as follows: 
\begin{align*}
 \Kcal_{\alpha}(x,y) = \sum_{r=0}^{\alpha}\frac{B_r(x)B_r(y)}{(r!)^2}+(-1)^{\alpha+1}\frac{B_{2\alpha}(|x-y|)}{(2\alpha)!} ,
\end{align*}
for $x,y\in [0,1)$, where $B_r$ denotes the Bernoulli polynomial of degree $r$, and
\begin{align*}
 \langle f, g \rangle_{\alpha} = \sum_{r=0}^{\alpha-1}\int_{0}^{1}f^{(r)}(x)\, \rd x \int_{0}^{1}g^{(r)}(x)\, \rd x + \int_{0}^{1}f^{(\alpha)}(x)g^{(\alpha)}(x)\, \rd x,
\end{align*}
for $f,g\in \Hcal_{\alpha}$.

In the $s$-variate case, we consider the $s$-fold tensor product space of the one-dimensional space introduced above.
That is, the Sobolev space $\Hcal_{\alpha,s}$ which we consider is simply given by $\Hcal_{\alpha,s}=\bigotimes_{j=1}^{s} \Hcal_{\alpha}$.
Then it is known from \cite[Section~8]{Aro50} that the reproducing kernel of the space $\Hcal_{\alpha,s}$ is the product of the reproducing kernels for the one-dimensional space $\Hcal_{\alpha}$.
Therefore, $\Hcal_{\alpha,s}$ is the reproducing kernel Hilbert space whose reproducing kernel $\Kcal_{\alpha,s}\colon [0,1)^s\times [0,1)^s\to \RR$ and inner product $\langle \cdot, \cdot \rangle_{\alpha,s}$ are given as follows:
\begin{align*}
 \Kcal_{\alpha,s}(\bsx,\bsy) = \prod_{j=1}^{s}\Kcal_{\alpha}(x_j,y_j) ,
\end{align*}
for $\bsx=(x_1,\ldots,x_s),\bsy=(y_1,\ldots,y_s)\in [0,1)^s$, and
\begin{align*}
 \langle f, g \rangle_{\alpha,s} & = \sum_{u\subseteq \{1,\ldots,s\}}\sum_{\bsr_u\in \{0,\ldots,\alpha-1\}^{|u|}} \int_{[0,1)^{s-|u|}} \\
 & \qquad \left(\int_{[0,1)^{|u|}}f^{(\bsr_u,\bsalpha)}(\bsx)\, \rd \bsx_u\right) \left(\int_{[0,1)^{|u|}} g^{(\bsr_u,\bsalpha)}(\bsx) \, \rd \bsx_u\right) \, \rd \bsx_{\{1,\ldots,s\}\setminus u} ,
\end{align*}
for $f,g\in \Hcal_{\alpha,s}$, where we use the following notation: For $u\subseteq \{1,\ldots,s\}$ and $\bsx\in [0,1)^s$, we write $\bsx_u=(x_j)_{j\in u}$.
Moreover, for $\bsr_u=(r_j)_{j\in u}\in \{0,\ldots,\alpha-1\}^{|u|}$, $(\bsr_u,\bsalpha)$ denotes the $s$-dimensional vector whose $j$-th component is $r_j$ if $j\in u$, and $\alpha$ otherwise.
Note that an integral and sum over the empty set is the identity operator.

Note that the so-called weight parameters, or more simply the weights, are not taken account of in the definition of $\Hcal_{\alpha,s}$.
As in \cite{NWbook,SW98}, the weights moderate the importance of different variables or groups of variables in function spaces and play an important role in the study of tractability.
However, such an investigation is out of the scope of this paper since we are interested in showing the optimal exponent of $\log N$ term in the error bound.

%%%%%%%%%%%%%%%%%%%%%%%%%%%%%%%%%%%%%%%%%%%%%%%%%%%%%%%%%%%%%%%%%%%%%%%%%%%%%%%%%%%%%%%%%%%%%%%%%%%
\subsection{Walsh functions}
The system of Walsh functions is the key tool for the error analysis of digital nets.
We refer to \cite[Appendix~A]{DPbook} for comprehensive information on Walsh functions in the context of QMC integration.
First let us define the one-dimensional Walsh functions.

\begin{definition}
For $b\in \NN$, $b\geq 2$, let $\omega_b:=\exp(2\pi \sqrt{-1}/b)$. For $k\in \NN_0$, we denote its $b$-adic expansion by $k=\kappa_0+\kappa_1 b+\cdots$, which is actually a finite expansion. The $k$-th $b$-adic Walsh function ${}_b\wal_k\colon [0,1)\to \{1,\omega_b,\ldots,\omega_b^{b-1}\}$ is defined by
\begin{align*}
{}_b\wal_k(x) := \omega_b^{\kappa_0 \xi_1+\kappa_1 \xi_2+\cdots},
\end{align*}
where we denote the $b$-adic expansion of $x\in [0,1)$ by $x=\sum_{i=1}^{\infty}\xi_ib^{-i}$ with $\xi_i\in \FF_b$, which is understood to be unique in the sense that infinitely many of the $\xi_i$'s are different from $b-1$.

\end{definition}
\noindent 
The above definition can be extended to the high-dimensional case as follows.
\begin{definition}
For $b\in \NN$, $b\geq 2$ and $\bsk=(k_1,\ldots,k_s)\in \NN_0^s$, the $\bsk$-th $b$-adic Walsh function ${}_b\wal_{\bsk}\colon [0,1)^s\to \{1,\omega_b,\ldots,\omega_b^{b-1}\}$ is defined by
\begin{align*}
{}_b\wal_{\bsk}(\bsx) := \prod_{j=1}^{s}{}_b\wal_{k_j}(x_j) .
\end{align*}
\end{definition}
\noindent
Since we shall always use Walsh functions in a fixed prime base $b$, we omit the subscript and simply write $\wal_k$ or $\wal_{\bsk}$.

As in \cite[Theorem~A.11]{DPbook}, the Walsh system $\{\wal_{\bsk}\colon \bsk\in \NN_0^s\}$ is a complete orthonormal system in $L^2([0,1)^s)$ for any $s\in \NN$.
Thus, we can define the Walsh series of $f\in L^2([0,1)^s)$ by
\begin{align*}
\sum_{\bsk\in \NN_0^s}\hat{f}(\bsk)\wal_{\bsk}(\bsx) ,
\end{align*}
where $\hat{f}(\bsk)$ denotes the $\bsk$-th Walsh coefficient of $f$ defined by
\begin{align*}
\hat{f}(\bsk) := \int_{[0,1)^s}f(\bsx)\overline{\wal_{\bsk}(\bsx)}\rd \bsx.
\end{align*}
We refer to \cite[Appendix~A.3]{DPbook} and \cite[Lemma~18]{GSYtent1} for a discussion on the pointwise absolute convergence of the Walsh series.
In fact, regarding the reproducing kernel $\Kcal_{\alpha,s}$ introduced before, we have the pointwise absolute convergence of the Walsh series, i.e., we have
\begin{align*}
 \Kcal_{\alpha,s}(\bsx,\bsy) = \sum_{\bsk,\bsl\in \NN_0^s}\hat{\Kcal}_{\alpha,s}(\bsk,\bsl)\wal_{\bsk}(\bsx)\overline{\wal_{\bsl}(\bsy)},
\end{align*}
for any $\bsx,\bsy\in [0,1)^s$, where we define
\begin{align*}
 \hat{\Kcal}_{\alpha,s}(\bsk,\bsl) := \int_{[0,1)^{2s}}\Kcal_{\alpha,s}(\bsx,\bsy)\overline{\wal_{\bsk}(\bsx)}\wal_{\bsl}(\bsy) \rd \bsx \rd \bsy ,
\end{align*}
for $\bsk,\bsl\in \NN_0^s$.

%%%%%%%%%%%%%%%%%%%%%%%%%%%%%%%%%%%%%%%%%%%%%%%%%%%%%%%%%%%%%%%%%%%%%%%%%%%%%%%%%%%%%%%%%%%%%%%%%%%
%%%%%%%%%%%%%%%%%%%%%%%%%%%%%%%%%%%%%%%%%%%%%%%%%%%%%%%%%%%%%%%%%%%%%%%%%%%%%%%%%%%%%%%%%%%%%%%%%%%
\section{Explicit construction of point sets}\label{sec:constr}
In this section, we first recall a construction of digital nets due to Chen and Skriganov and digit interlacing composition due to Dick in Subsections~\ref{subsec:cs_constr} and \ref{subsec:interlace}, respectively.
Then we provide an explicit construction of point sets which achieve the best possible convergence in $\Hcal_{\alpha,s}$.

%%%%%%%%%%%%%%%%%%%%%%%%%%%%%%%%%%%%%%%%%%%%%%%%%%%%%%%%%%%%%%%%%%%%%%%%%%%%%%%%%%%%%%%%%%%%%%%%%%%
\subsection{Chen and Skriganov's construction}\label{subsec:cs_constr}
The key property of Chen-Skriganov's digital nets is that both the minimum Hamming and NRT metrics are large enough simultaneously.
In what follows we first introduce a construction of digital nets due to Chen and Skriganov \cite{CS02,Skr06} by following the exposition of \cite[Section~16.4]{DPbook} and then provide its key property.

Let $s,g,w\in \NN$ and $b\geq gs$ be a prime.
Let $\{\beta_{j,l}\}_{1\leq j\leq s, 1\leq l\leq g}$ with $\beta_{j,l}\in \FF_b$ be a set of $gs$ distinct elements.
An explicit construction of digital nets over $\FF_b$ consisting of $b^{gw}$ points due to Chen and Skriganov is given by the generating matrices $C_1,\ldots,C_s\in \FF_b^{gw\times gw}$ which are defined by
\begin{align*}
C_j = (c^{(j)}_{u,v})_{u,v=1,\ldots,gw} ,
\end{align*}
where 
\begin{align*}
c^{(j)}_{(l-1)w+i,v} = \binom{v-1}{i-1}\beta_{j,l}^{v-i} ,
\end{align*}
for $1\leq i\leq w$, $1\leq l\leq g$, $1\leq v\leq gw$ and $1\leq j\leq s$.
In the above, $\binom{i}{j}$ denotes a binomial coefficient modulo $b$ and we use the usual conventions that $\binom{i}{j}=0$ whenever $i<j$ and that $0^0=1$.
This construction can be regarded as a generalization of the construction of digital nets due to Faure \cite{Faure82}, who studied the case $g=1$.

The following key property is obtained in \cite[Lemma~2E]{CS02}, see also \cite[Theorem~16.28]{DPbook}.
\begin{lemma}\label{lem:cs_dual}
Let $s,g,w\in \NN$ and $b\geq gs$ be a prime.
Let $P$ be a digital net over $\FF_b$ consisting of $b^{gw}$ points due to Chen and Skriganov.
Then we have
\begin{align*}
\varkappa(P) \geq g+1\quad \text{and}\quad \mu_1(P)\geq gw+1 .
\end{align*}
\end{lemma}
%%%%%%%%%%%%%%%%%%%%%%%%%%%%%%%%%%%%%%%%%%%%%%%%%%%%%%%%%%%%%%%%%%%%%%%%%%%%%%%%%%%%%%%%%%%%%%%%%%%
\subsection{Dick's digit interlacing composition}\label{subsec:interlace}
Here we recall digit interlacing composition due to Dick \cite{Dick08,Dick09} to construct digital nets with large minimum $\mu_{\alpha}$ metric (for a given integer $\alpha\geq 2$) from digital nets with large minimum NRT metric.
First we introduce the definition of order $\alpha$ digital $(t,m,s)$-net over $\FF_b$ for $\alpha \geq 1$.

\begin{definition}
For a prime $b$ and $s,\alpha,m,n\in \NN$ with $n\geq \alpha m$, let $P$ be a digital net over $\FF_b$ with generating matrices $C_1,\ldots,C_s\in \FF_b^{n\times m}$.
We denote by $\bsc_{i,j}\in \FF_b^m$ the $i$-th row vector of $C_j$ for $1\leq i\leq n$ and $1\leq j\leq s$.
Let $t$ be an integer with $0\leq t\leq \alpha m$ which satisfies the following condition:
For all $1\leq i_{j,v_j}<\cdots <i_{j,1}\leq n$ such that
\begin{align*}
\sum_{j=1}^{s} \sum_{l=1}^{\min(\alpha, v_j)} i_{j,l} \leq \alpha m -t ,
\end{align*}
the vectors $\bsc_{i_{1,v_1},1},\ldots,\bsc_{i_{1,1},1},\ldots,\bsc_{i_{s,v_s},s},\ldots,\bsc_{i_{s,1},s}$ are linearly independent over $\FF_b$.
Then we call $P$ an order $\alpha$ digital $(t,m,s)$-net over $\FF_b$.
\end{definition}
\noindent 
It follows from the above linear independence of the rows of generating matrices that any order $\alpha$ digital $(t,m,s)$-net $P$ over $\FF_b$ satisfies
\begin{align*}
\mu_{\alpha}(P) \geq \alpha m -t +1.
\end{align*}
Therefore, order $\alpha$ digital $(t,m,s)$-nets with small $t$-value are exactly digital nets with large minimum $\mu_{\alpha}$ metric.
There are many explicit constructions of order 1 digital $(t,m,s)$-nets with small $t$-value for an arbitrary dimension $s$.
We again refer to \cite{Faure82,Nied88,NXbook,Sobol67} as well as \cite[Chapter~8]{DPbook} on such constructions.
Note that, in this light, Chen-Skriganov's digital nets can be seen as order 1 digital $(0,gw,s)$-nets over $\FF_b$ for $g,w\in \NN$ and a prime $b\geq gs$.

In order to construct order $\alpha$ digital $(t,m,s)$-nets with small $t$-value, we now introduce the digit interlacing composition due to Dick:
For a prime $b$ and $s,\alpha,m\in \NN$ with $\alpha \geq 2$, let $Q\subset [0,1)^{\alpha s}$ be a digital net over $\FF_b$ with generating matrices $C_1,\ldots,C_{\alpha s}\in \FF_b^{m\times m}$.
We denote by $\bsc_{i,j}\in \FF_b^m$ the $i$-th row vector of $C_j$ for $1\leq i\leq m$ and $1\leq j\leq \alpha s$.
Now we construct a digital net $P\subset [0,1)^s$ with generating matrices $D_1,\ldots,D_s\in \FF_b^{\alpha m\times m}$ such that the $(\alpha(h-1)+i)$-th row vector of $D_j$ equals $\bsc_{h,\alpha (j-1)+i}$ for $1\leq h\leq m$, $1\leq i\leq \alpha$ and $1\leq j\leq s$.
The key property of this construction algorithm is given as follows, see for instance \cite[Corollary~3.4]{BDP11}.

\begin{lemma}\label{lem:dick_dual}
Let $Q$ be an order 1 digital $(t',m,\alpha s)$-net over $\FF_b$ with $0\leq t'\leq m$.
Then a digital net $P$ constructed as above is an order $\alpha$ digital $(t,m,s)$-net over $\FF_b$ with
\begin{align*}
t = \alpha \min \left\{m, t'+ \left\lfloor \frac{s(\alpha-1)}{2} \right\rfloor\right\} .
\end{align*}
Therefore, $P$ satisfies
\begin{align*}
\mu_{\alpha}(P) \geq \max \left\{0, \alpha \left( m -t'- \left\lfloor \frac{s(\alpha-1)}{2} \right\rfloor\right)\right\}+1.
\end{align*}
\end{lemma}

Further, we need the so-called propagation property shown in \cite[Theorem~3.3]{Dick07} and \cite[Theorem~4.10]{Dick08} as follows:
\begin{lemma}\label{lem:propagation}
Let $1\leq \alpha'< \alpha$ and $0\leq t\leq \alpha m$.
Any order $\alpha$ digital $(t,m,s)$-net over $\FF_b$ is also an order $\alpha'$ digital $(t',m,s)$-net over $\FF_b$ with $$t'=\lceil t\alpha'/\alpha \rceil.$$
\end{lemma}
%%%%%%%%%%%%%%%%%%%%%%%%%%%%%%%%%%%%%%%%%%%%%%%%%%%%%%%%%%%%%%%%%%%%%%%%%%%%%%%%%%%%%%%%%%%%%%%%%%%
\subsection{Our explicit construction}\label{subsec:construction}
As we already mentioned in the first section, our explicit construction of point sets is simply given by applying Dick's digit interlacing composition to Chen-Skriganov's digital nets.
This is done as follows.

Let $s,\beta,g\in \NN$ and $b\geq \beta g s$ be a prime.
Let $\{\beta_{j,l}\}_{1\leq j\leq \beta s, 1\leq l\leq g}$ with $\beta_{j,l}\in \FF_b$ be a set of distinct $\beta gs$ elements.
For $w\in \NN$, we first construct a digital net $Q\subset [0,1)^{\beta s}$ over $\FF_b$ with generating matrices $C_1,\ldots,C_{\beta s}\in \FF_b^{gw\times gw}$ which are given by
\begin{align*}
C_j = (c^{(j)}_{u,v})_{u,v=1,\ldots,gw} ,
\end{align*}
where 
\begin{align*}
c^{(j)}_{(l-1)w+i,v} = \binom{v-1}{i-1}\beta_{j,l}^{v-i} ,
\end{align*}
for $1\leq i\leq w$, $1\leq l\leq g$, $1\leq v\leq gw$ and $1\leq j\leq \beta s$.
Then we construct a digital net $P\subset [0,1)^s$ over $\FF_b$ with generating matrices $D_1,\ldots,D_s\in \FF_b^{\beta gw\times gw}$ by applying the digit interlacing composition to $Q$, that is, $D_1,\ldots,D_s$ are given such that the $(\beta(h-1)+i)$-th row vector of $D_j$ equals $\bsc_{h,\beta (j-1)+i}$ for $1\leq h\leq gw$, $1\leq i\leq \beta$ and $1\leq j\leq s$.

Note that a digital net $P$ constructed as above consists of $b^{gw}$ points for $w\in \NN$. The following important property of $P$ will be crucial in the proof of our main result.
\begin{lemma}\label{lem:gsy_constr}
Let $P$ be a digital net in $[0,1)^s$ constructed as above. Then we have
\begin{align*}
\varkappa(P) \geq g+1 \quad \text{and} \quad \mu_{\beta}(P) \geq \max \left\{0, \beta \left( gw -\left\lfloor \frac{s(\beta-1)}{2} \right\rfloor\right)\right\}+1.
\end{align*}
\end{lemma}
\begin{proof}
Let $Q$ be a digital net in $[0,1)^{\beta s}$ constructed as above.
Since $Q$ is nothing but the Chen-Skriganov's digital net, it follows that $Q$ is an order 1 digital $(0,gw,\beta s)$-net, and moreover from Lemma~\ref{lem:cs_dual} we have
\begin{align*}
\varkappa(Q) \geq g+1\quad \text{and}\quad \mu_1(Q)\geq gw+1 .
\end{align*}

First we prove $\varkappa(P)=\varkappa(Q)$, from which the result for the first part follows.
Let $\Ecal_{\beta}: \NN_0^{\beta}\to \NN_0$ be defined by
\begin{align*}
\Ecal_{\beta}(k_1,\ldots,k_{\beta}) := \sum_{a=0}^{\infty}\sum_{j=1}^{\beta}\kappa_{a,j}b^{a\beta +j-1} ,
\end{align*}
where we denote the $b$-adic expansion of $k_j$ by $k_j=\kappa_{0,j}+\kappa_{1,j}b+\cdots$, which is actually a finite expansion.
It is obvious that $\Ecal_{\beta}$ is a bijection between the set $\{0,\ldots,b^m-1\}^{\beta}$ and the set $\{0,\ldots,b^{\beta m}-1\}$ for any $m\in \NN$.
For any $(k_1,\ldots,k_{\beta}) \in \{0,\ldots,b^{gw}-1\}^{\beta}$ we have
\begin{align*}
& \quad C_1^{\top}\tr_{gw}(k_1)\oplus \cdots \oplus C_{\beta}^{\top}\tr_{gw}(k_{\beta}) \\
& = (\bsc_{1,1}^{\top},\ldots,\bsc_{gw,1}^{\top}) (\kappa_{0,1},\kappa_{1,1},\ldots,\kappa_{gw-1,1})^{\top} \\
& \qquad \oplus \cdots \oplus  (\bsc_{1,\beta}^{\top},\ldots,\bsc_{gw,\beta}^{\top}) (\kappa_{0,\beta},\kappa_{1,\beta},\ldots, \kappa_{gw-1,\beta})^{\top}  \\
& = (\bsc_{1,1}^{\top},\bsc_{1,2}^{\top},\ldots,\bsc_{1,\beta}^{\top},\ldots, \bsc_{gw,1}^{\top},\bsc_{gw,2}^{\top},\ldots,\bsc_{gw,\beta}^{\top}) \\
& \qquad \cdot (\kappa_{0,1},\kappa_{0,2},\ldots,\kappa_{0,\beta},\ldots,\kappa_{gw-1,1},\kappa_{gw-1,2},\ldots,\kappa_{gw-1,\beta})^{\top}  \\
& = D_1^{\top}\tr_{\beta gw}\left( \Ecal_{\beta}(k_1,\ldots,k_{\beta})\right) ,
\end{align*}
where the last equality stems from the digit interlacing composition of $D_1$ and the definition of $\Ecal_{\beta}$.
In case of vectors in $\NN_0^{\beta s}$, we apply $\Ecal_{\beta}$ to every non-overlapping block of consecutive $\beta$ components, i.e., we define
\begin{align*}
\Ecal_{\beta}(k_1,\ldots,k_{\beta s}) := (\Ecal_{\beta}(k_1,\ldots,k_{\beta}),\ldots,\Ecal_{\beta}(k_{\beta(s-1)+1},\ldots,k_{\beta s})) \in \NN_0^{s}.
\end{align*}
Again it is obvious that $\Ecal_{\beta}$ is a bijection between the set $\{0,\ldots,b^m-1\}^{\beta s}$ and the set $\{0,\ldots,b^{\beta m}-1\}^s$ for any $m\in \NN$.
Then using the above result we have
\begin{align*}
& \quad C_1^{\top}\tr_{gw}(k_1)\oplus \cdots \oplus C_{\beta s}^{\top}\tr_{gw}(k_{\beta s}) \\
& = D_1^{\top}\tr_{\beta gw}\left( \Ecal_{\beta}(k_1,\ldots,k_{\beta})\right)\oplus \cdots \oplus D_s^{\top}\tr_{\beta gw}\left( \Ecal_{\beta}(k_{\beta(s-1)+1},\ldots,k_{\beta s})\right) ,
\end{align*}
for any $(k_1,\ldots,k_{\beta s}) \in \{0,\ldots,b^{gw}-1\}^{\beta s}$, from which it follows that
\begin{align*}
P^{\perp} = \left\{ \Ecal_{\beta}(\bsk)\colon \bsk\in Q^{\perp}\right\} =: \Ecal_{\beta}(Q^{\perp}) .
\end{align*}
It is also obvious that for any $\bsk=(k_1,\ldots,k_{\beta s}) \in \{0,\ldots,b^{gw}-1\}^{\beta s}$ we have
\begin{align*}
\varkappa(\bsk) = \varkappa(\Ecal_{\beta}(\bsk)) .
\end{align*}
Therefore it holds that
\begin{align*}
\varkappa(P) & = \min_{\bsk\in P^{\perp}\setminus \{\bszero\}}\varkappa(\bsk) = \min_{\bsk\in \Ecal_{\beta}(Q^{\perp})\setminus \{\bszero\}}\varkappa(\bsk) \\
& = \min_{\bsk\in Q^{\perp}\setminus \{\bszero\}}\varkappa(\Ecal_{\beta}(\bsk)) = \min_{\bsk\in Q^{\perp}\setminus \{\bszero\}}\varkappa(\bsk) = \varkappa(Q) ,
\end{align*}
which proves the first part of this lemma.

Using the fact that $Q$ is an order 1 digital $(0,gw,\beta s)$-net and Lemma~\ref{lem:dick_dual}, we can easily see that $P$ is an order $\beta$ digital $(t,gw,s)$-net with
\begin{align*}
t = \beta \min \left\{gw, \left\lfloor \frac{s(\beta-1)}{2} \right\rfloor\right\} .
\end{align*}
Therefore, $P$ satisfies
\begin{align*}
\mu_{\beta}(P) \geq \max \left\{0, \beta \left( gw - \left\lfloor \frac{s(\beta-1)}{2} \right\rfloor\right)\right\}+1,
\end{align*}
which proves the second part of this lemma.
\end{proof}

\begin{remark}\label{rem:gsy_constr}
Let $\alpha\in \NN$, $\alpha \geq 2$.
As already stated in Theorem~\ref{thm:main}, we need to set $\beta\geq 2\alpha$ and $g\geq 2\alpha s$ in order for $P$ to achieve the best possible rate of convergence in $\Hcal_{\alpha,s}$.
The condition $\beta\geq 2\alpha$ is required to exploit the decay of the Walsh coefficients of $\Kcal_{\alpha,s}$ in a suitable manner as done in \cite{GSYexist}, 
whereas the condition $g\geq 2\alpha s$ is to exploit the sparsity of the Walsh coefficients, which shall be made clear in the next section.
The additional condition $g\geq \lfloor s(\beta -1)/2\rfloor$ in Theorem~\ref{thm:main} is included just for a trivial technical reason such that the $t$-value of $P$ as an order $\beta$ digital net is independent of the choice $w\in \NN$.
\end{remark}

%%%%%%%%%%%%%%%%%%%%%%%%%%%%%%%%%%%%%%%%%%%%%%%%%%%%%%%%%%%%%%%%%%%%%%%%%%%%%%%%%%%%%%%%%%%%%%%%%%%
%%%%%%%%%%%%%%%%%%%%%%%%%%%%%%%%%%%%%%%%%%%%%%%%%%%%%%%%%%%%%%%%%%%%%%%%%%%%%%%%%%%%%%%%%%%%%%%%%%%
\section{The proof of the main result}\label{sec:proof}
We first provide the proof of Theorem~\ref{thm:main} by using the results which shall be shown later in Subsections~\ref{subsec:sparsity}, \ref{subsec:main_part} and \ref{subsec:disc_part}.
\begin{proof}[Proof of Theorem~\ref{thm:main}]
Here we only consider the case $\alpha \geq 3$, although the case $\alpha=2$ can be shown in a similar way.
For any digital net $P$ over $\FF_b$ with generating matrices of the size $n\times m$, we have
\begin{align}\label{eq:main_discre}
& \left(e^{\wor}(\Hcal_{\alpha,s};P)\right)^2 \nonumber \\
& \quad = \sum_{\bsk,\bsl\in P_{\infty}^{\perp}\setminus \{\bszero\}}\hat{\Kcal}_{\alpha,s}(\bsk,\bsl) \leq \sum_{\bsk,\bsl\in P_{\infty}^{\perp}\setminus \{\bszero\}}\left|\hat{\Kcal}_{\alpha,s}(\bsk,\bsl)\right| \nonumber \\
& \quad \leq \sum_{\bsk,\bsl\in P^{\perp}\setminus \{\bszero\}}\left|\hat{\Kcal}_{\alpha,s}(\bsk,\bsl)\right| + \sum_{\substack{\bsk\in P_{\infty}^{\perp}\setminus P^{\perp} \\ \bsl\in P_{\infty}^{\perp}\setminus \{\bszero\}}}\left|\hat{\Kcal}_{\alpha,s}(\bsk,\bsl)\right| + \sum_{\substack{\bsk\in P_{\infty}^{\perp}\setminus \{\bszero\} \\ \bsl\in P_{\infty}^{\perp}\setminus P^{\perp}}}\left|\hat{\Kcal}_{\alpha,s}(\bsk,\bsl)\right| \nonumber \\
& \quad = \sum_{\bsk,\bsl\in P^{\perp}\setminus \{\bszero\}}\left|\hat{\Kcal}_{\alpha,s}(\bsk,\bsl)\right| + 2\sum_{\substack{\bsk\in P_{\infty}^{\perp}\setminus P^{\perp} \\ \bsl\in P_{\infty}^{\perp}\setminus \{\bszero\}}}\left|\hat{\Kcal}_{\alpha,s}(\bsk,\bsl)\right| \nonumber \\
& \quad \leq \sum_{\bsk,\bsl\in P^{\perp}\setminus \{\bszero\}}\left|\hat{\Kcal}_{\alpha,s}(\bsk,\bsl)\right| + 2\sum_{\substack{\bsk\in \NN_0^s \setminus \{0,1,\ldots,b^n-1\}^s \\ \bsl\in P_{\infty}^{\perp}\setminus \{\bszero\}}}\left|\hat{\Kcal}_{\alpha,s}(\bsk,\bsl)\right|
\end{align}
where the first equality is given in \cite[Proof of Theorem~15]{BD09}, the second equality stems from the property $\overline{\hat{\Kcal}_{\alpha,s}(\bsk,\bsl)}=\hat{\Kcal}_{\alpha,s}(\bsl,\bsk)$ for any $\bsk,\bsl\in \NN_0^s$, and both the second and last inequalities follow immediately from the definitions of dual nets $P^{\perp}$ and $P^{\perp}_{\infty}$.
In what follows, the first and second terms of (\ref{eq:main_discre}) are called the \emph{main part} and \emph{discretization part} of the squared worst-case error, respectively.

Now let $\beta,g\in \NN$ be given such that $\beta \geq 2\alpha$, $g\geq 2\alpha s$ and $g\geq \lfloor s(\beta -1)/2\rfloor$. Moreover let $b\geq \beta gs$ be a prime.
For $w\in \NN$, let $P$ be constructed as in Subsection~\ref{subsec:construction} and let $N = |P| = b^{gw}$.
By using the upper bounds on the main part and the discretization part for $P$ shown in Propositions~\ref{prop:main_part} and \ref{prop:discre_part}, respectively, we see that the discretization part does not affect the order of convergence appearing in the main part.
That is, we have
\begin{align*}
\left(e^{\wor}(\Hcal_{\alpha,s};P)\right)^2 & \leq A^{(1)}_{\alpha,\beta,b,s}\frac{(\log N)^{s-1}}{N^{2\alpha}} + O\left(\frac{(\log N)^{s\alpha}}{N^{3\alpha}}\right) .
\end{align*}
Thus there exists a positive constant $C_{\alpha,\beta,b,s}$ such that Theorem~\ref{thm:main} holds.
\end{proof}

%%%%%%%%%%%%%%%%%%%%%%%%%%%%%%%%%%%%%%%%%%%%%%%%%%%%%%%%%%%%%%%%%%%%%%%%%%%%%%%%%%%%%%%%%%%%%%%%%%%
\subsection{Sparsity of the Walsh coefficients}\label{subsec:sparsity}
Let us consider the one-dimensional case first.
In the following, we write
\begin{align*}
\hat{b}_{r}(k) = \int_0^1 \frac{B_{r}(x)}{r !}\overline{\wal_k(x)}\rd x ,
\end{align*}
and 
\begin{align*}
\hat{b}_{r,\per}(k,l) = \int_0^1\int_0^1 \frac{\tilde{B}_{r}(x-y)}{r !}\overline{\wal_k(x)}\wal_l(y)\rd x \rd y,
\end{align*}
for $r\in \NN_0$ and $k,l\in \NN_0$, where $\tilde{B}_{r}:\RR\to \RR$ is defined by extending $B_{r}$ periodically to $\RR$.
Note that we have $B_{r}(|x-y|)=\tilde{B}_{r}(x-y)$ for even $r$ and $B_{r}(|x-y|)=(-1)^{1_{x<y}}\tilde{B}_{r}(x-y)$ for odd $r$ for any $x,y\in [0,1)$, where $1_{x<y}$ equals 1 if $x<y$ and 0 otherwise.
Then it is obvious that
\begin{align}\label{eq:sparse_walsh}
\hat{\Kcal}_{\alpha}(k,l) = \sum_{\tau=0}^{\alpha}\hat{b}_{\tau}(k)\overline{\hat{b}_{\tau}(l)}+(-1)^{\alpha+1}\hat{b}_{2\alpha,\per}(k,l) .
\end{align}
In the following proposition, we show that $\hat{\Kcal}_{\alpha}(k,l)=0$ for many choices of $k,l\in \NN_0$, which means that the Walsh coefficients $\hat{\Kcal}_{\alpha}$ are actually sparse.
\begin{proposition}\label{prop:sparse_walsh}
For $\alpha\in \NN$, $\alpha\geq 2$, let $k,l\in \NN_0$ with $\varkappa(k\ominus l)> 2\alpha $. Then we have $\hat{\Kcal}_{\alpha}(k,l)=0$.
\end{proposition}

In order to prove Proposition~\ref{prop:sparse_walsh}, it suffices to show that every term on the right-hand side of (\ref{eq:sparse_walsh}) is 0 whenever $\varkappa(k\ominus l)> 2\alpha $, which shall be proven in Lemmas~\ref{lem:sparse_walsh_1} and \ref{lem:sparse_walsh_2} below.
Throughout this subsection, we denote the $b$-adic expansions of $k,l\in \NN$ by $k=\kappa_1 b^{a_1-1}+\cdots+\kappa_v b^{a_v-1}$ and $l=\lambda_1 b^{d_1-1}+\cdots+\lambda_w b^{d_w-1}$, respectively, where $\kappa_1,\ldots,\kappa_v,\lambda_1,\ldots,\lambda_w\in \{1,\ldots,b-1\}$, $a_1>\cdots>a_v>0$ and $d_1>\cdots>d_w>0$.
When $k=0$ ($l=0$, resp.), we assume that $v=0$ and $a_0=0$ ($w=0$ and $d_0=0$, resp.).
Note that we have $\varkappa(k)=v$ and $\varkappa(l)=w$.

\begin{lemma}\label{lem:sparse_walsh_1}
For $\alpha\in \NN$, $\alpha\geq 2$, let $k,l\in \NN_0$ with $\varkappa(k\ominus l)> 2\alpha $. Then we have $\hat{b}_{\tau}(k)\overline{\hat{b}_{\tau}(l)}=0$ for any $0\leq \tau\leq \alpha$.
\end{lemma}
\begin{proof}
Since $\varkappa(k\ominus l)\leq \varkappa(k)+\varkappa(l)=v+w$, we must have either $v> \alpha$ or $w> \alpha$ whenever $\varkappa(k\ominus l)> 2\alpha $. Then it follows from \cite[Section~4]{Dick09} that either $\hat{b}_{\tau}(k)=0$ or $\hat{b}_{\tau}(l)=0$ for all $0\leq \tau\leq \alpha$, which completes the proof.
\end{proof}

\begin{lemma}\label{lem:sparse_walsh_2}
For $r\in \NN$, $r\geq 2$, let $k,l\in \NN_0$ with $\varkappa(k\ominus l)> r $. Then we have $\hat{b}_{r,\per}(k,l)=0$.
\end{lemma}
\begin{proof}
We prove this lemma by induction on $r\geq 2$. For $k,l\in \NN_0$, we shall write $k'=k-\kappa_1 b^{a_1-1}$, $k''=k'-\kappa_2 b^{a_2-1}$, $l'=l-\lambda_1 b^{d_1-1}$ and $l''=l'-\lambda_2 b^{d_2-1}$.

Let us consider the case $r=2$ first.
For $k,l\in \NN_0$ with $\varkappa(k\ominus l)> 2$, it never follows that $k=l$, $k'=l'$ with $k\neq l$, $k'=l$, $k=l'$, $k''=l$, or $k=l''$, since we have $\varkappa(k\ominus l)=$ 0, 2, 1, 1, 2 and 2, respectively.
Then the result immediately follows from \cite[Lemma~10]{Dick09}.

Now assume that the result holds true for $r-1$. That is, we assume that 
\begin{align}\label{eq:sparse_walsh_proof_1}
\hat{b}_{r-1,\per}(k,l)=0 \quad \text{for $k,l\in \NN_0$ with $\varkappa(k\ominus l)> r-1$}.
\end{align}
If either $k=0$ or $l=0$ holds, the result $\hat{b}_{r,\per}(k,l)=0$ immediately follows from \cite[Lemma~11]{Dick09}. Thus we focus on the case $k,l>0$ in the following. As in \cite[Equation~14.18]{DPbook}, for any $k,l\in \NN$ and $r>2$ we have the identity
\begin{align*}
\hat{b}_{r,\per}(k,l) & = -\frac{1}{b^{a_1}}\Big(\frac{1}{1-\omega_b^{-\kappa_1}}\hat{b}_{r-1,\per}(k',l)+\left( \frac{1}{2}+\frac{1}{\omega_b^{-\kappa_1}-1}\right)\hat{b}_{r-1,\per}(k,l) \\
& \qquad \qquad + \sum_{c=1}^{\infty}\sum_{\theta=1}^{b-1}\frac{1}{b^c(\omega_b^{\theta}-1)}\hat{b}_{r-1,\per}(\theta b^{c+a_1-1}+ k,l) \Big) .
\end{align*}
Thus, in order to prove $\hat{b}_{r,\per}(k,l)=0$ for $k,l\in \NN$ with $\varkappa(k\ominus l)> r $, it suffices to prove that (i) $\hat{b}_{r-1,\per}(k',l)=0$, (ii) $\hat{b}_{r-1,\per}(k,l)=0$ and (iii) $\hat{b}_{r-1,\per}(\theta b^{c+a_1-1}+ k,l)=0$ for any $c\in \NN$ and $1\leq \theta\leq b-1$ whenever $\varkappa(k\ominus l)> r $.
Here, from the assumption (\ref{eq:sparse_walsh_proof_1}), it is trivial that $\hat{b}_{r-1,\per}(k,l)=0$ also for $k,l\in \NN$ with $\varkappa(k\ominus l)> r$. The remaining two items (i) and (iii) can be proven in the following way.

Since we have
\begin{align*}
\varkappa(k\ominus l) = \varkappa(k'\ominus l\oplus \kappa_1b^{a_1-1}) & \leq \varkappa(k'\ominus l)+\varkappa(\kappa_1b^{a_1-1}) \\
& = \varkappa(k'\ominus l)+1,
\end{align*}
it holds that $\varkappa(k'\ominus l)\geq \varkappa(k\ominus l)-1>r-1$. Thus, again from the assumption (\ref{eq:sparse_walsh_proof_1}), it follows that $\hat{b}_{r-1,\per}(k',l)=0$, which completes the proof of the item (i).

Similarly, since we have
\begin{align*}
\varkappa(k\ominus l) & = \varkappa(k\ominus l\oplus \theta b^{c+a_1-1}\ominus \theta b^{c+a_1-1}) \\
& \leq \varkappa(k\ominus l\oplus \theta b^{c+a_1-1}) + \varkappa(\theta b^{c+a_1-1}) \\
& = \varkappa((\theta b^{c+a_1-1}+k)\ominus l) + 1,
\end{align*}
for any $c\in \NN$ and $1\leq \theta\leq b-1$, it holds that $\varkappa((\theta b^{c+a_1-1}+k)\ominus l)\geq \varkappa(k\ominus l)-1>r-1$. Again from the assumption (\ref{eq:sparse_walsh_proof_1}), it follows that $\hat{b}_{r-1,\per}(\theta b^{c+a_1-1}+ k,l)=0$, which completes the proof of the item (iii).
\end{proof}

Let us move on to the high-dimensional case. As a corollary of Proposition~\ref{prop:sparse_walsh} we have the following, which shows the sparsity of the Walsh coefficients $\hat{\Kcal}_{\alpha,s}$.
\begin{corollary}\label{cor:sparse_walsh}
For $s,\alpha\in \NN$, $\alpha\geq 2$, let $\bsk,\bsl\in \NN_0^s$ with $\varkappa(\bsk\ominus \bsl)> 2\alpha s$. Then we have $\hat{\Kcal}_{\alpha,s}(\bsk,\bsl)=0$.
\end{corollary}
\begin{proof}
From the definitions of $\Kcal_{\alpha,s}$ and Walsh functions, we have
\begin{align}\label{eq:walsh_coeff_projection}
\hat{\Kcal}_{\alpha,s}(\bsk,\bsl) & = \int_{[0,1)^{2s}}\prod_{j=1}^{s}\Kcal_{\alpha}(x_j,y_j)\overline{\wal_{k_j}(x_j)}\wal_{l_j}(y_j) \rd \bsx \rd \bsy \nonumber \\
& = \prod_{j=1}^{s}\int_{[0,1)^2}\Kcal_{\alpha}(x_j,y_j)\overline{\wal_{k_j}(x_j)}\wal_{l_j}(y_j) \rd x_j \rd y_j \nonumber \\
& = \prod_{j=1}^{s}\hat{\Kcal}_{\alpha}(k_j,l_j) ,
\end{align}
for any $\bsk,\bsl\in \NN_0^s$.
From the definition of $\varkappa$, it follows that $\varkappa(\bsk \ominus \bsl)=\sum_{j=1}^{s}\varkappa(k_j\ominus l_j)$.
From the assumption $\varkappa(\bsk\ominus \bsl)> 2\alpha s$ and the pigeonhole principle, it follows that there exists at least one index $j\in \{1,\ldots,s\}$ such that $\varkappa(k_j\ominus l_j)>2\alpha$.
For such a $j$, it follows from Proposition~\ref{prop:sparse_walsh} that $\hat{\Kcal}_{\alpha}(k_j,l_j)=0$, which completes the proof.
\end{proof}

%%%%%%%%%%%%%%%%%%%%%%%%%%%%%%%%%%%%%%%%%%%%%%%%%%%%%%%%%%%%%%%%%%%%%%%%%%%%%%%%%%%%%%%%%%%%%%%%%%%
\subsection{An upper bound on the main part}\label{subsec:main_part}
We recall that we have assumed that $P$ is a digital net constructed as in Subsection~\ref{subsec:construction} with $\beta\geq 2\alpha$ and $g\geq 2\alpha s$, and that
\begin{align*}
\varkappa(P) = \min_{\substack{\bsk,\bsl\in P^{\perp}\\ \bsk\neq \bsl}}\varkappa(\bsk\ominus \bsl).
\end{align*}
Since our explicit construction of $P$ gives $\varkappa(P)\geq g+1\geq 2\alpha s+1$ (see Lemma~\ref{lem:gsy_constr} and Remark~\ref{rem:gsy_constr}), it follows that $\varkappa(\bsk\ominus \bsl)>2\alpha s$ for any $\bsk,\bsl\in P^{\perp}$ with $\bsk\neq \bsl$.
Using this fact and Corollary~\ref{cor:sparse_walsh}, we have
\begin{align*}
\sum_{\bsk,\bsl\in P^{\perp}\setminus \{\bszero\}}\left|\hat{\Kcal}_{\alpha,s}(\bsk,\bsl)\right| & = \sum_{\bsk\in P^{\perp}\setminus \{\bszero\}}\left|\hat{\Kcal}_{\alpha,s}(\bsk,\bsk)\right| + \sum_{\substack{\bsk,\bsl\in P^{\perp}\setminus \{\bszero\}\\ \bsk\ne \bsl}}\left|\hat{\Kcal}_{\alpha,s}(\bsk,\bsl)\right|  \\
& = \sum_{\bsk\in P^{\perp}\setminus \{\bszero\}}\left|\hat{\Kcal}_{\alpha,s}(\bsk,\bsk)\right|.
\end{align*}
From Equation~(\ref{eq:walsh_coeff_projection}) and \cite[Proposition~20]{BD09}, each summand in the last expression can be bounded by
\begin{align*}
\left|\hat{\Kcal}_{\alpha,s}(\bsk,\bsk)\right| & = \prod_{j=1}^{s}\left|\hat{\Kcal}_{\alpha}(k_j,k_j)\right| \leq \prod_{j\colon k_j\neq 0}D_{\alpha,b}b^{-2\mu_{\alpha}(k_j)} \leq \max\{1,D_{\alpha,b}^{s}\}b^{-2\mu_{\alpha}(\bsk)} ,
\end{align*}
where $D_{\alpha,b}$ is positive and depends only on $\alpha$ and $b$.
Thus we have
\begin{align*}
\sum_{\bsk,\bsl\in P^{\perp}\setminus \{\bszero\}}\left|\hat{\Kcal}_{\alpha,s}(\bsk,\bsl)\right| & \leq \max\{1,D_{\alpha,b}^{s}\} \sum_{\bsk\in P^{\perp}\setminus \{\bszero\}}b^{-2\mu_{\alpha}(\bsk)}.
\end{align*}
In what follows, we shall show an upper bound on the last sum by following an argument similar to that used in \cite{GSYexist}.

By Lemma~\ref{lem:gsy_constr}, the digital net $P$ constructed as in Subsection~\ref{subsec:construction} is an order $\beta$ digital $(t,gw,s)$-net over $\FF_b$ with
\begin{align*}
t = \beta \min \left\{gw, \left\lfloor \frac{s(\beta-1)}{2} \right\rfloor\right\} = \beta \left\lfloor \frac{s(\beta-1)}{2} \right\rfloor,
\end{align*}
for any $w\in \NN$.
It is clear that the $t$-value of $P$ is independent of $w$.
It is also known from Lemma~\ref{lem:propagation} that $P$ is an order $1$ digital $(t',gw,s)$-net over $\FF_b$ with
\begin{align*}
t' = \lceil t/\beta \rceil = \left\lfloor \frac{s(\beta-1)}{2} \right\rfloor,
\end{align*}
which is again independent of $w$.
Therefore, $P$ satisfies
\begin{align*}
\mu_{\beta}(P) \geq \beta gw - t+1 \quad \text{and} \quad \mu_1(P) \geq gw - t'+1 .
\end{align*}
Moreover, in \cite[Lemma~3]{GSYexist}, the authors of this paper introduced an interpolation property of the Dick metric functions, that is, for any $1<\alpha\leq \beta$ and $\bsk\in \NN_0^s$, we have
\begin{align*}
\mu_{\alpha}(\bsk) \geq A_{\alpha \beta}\mu_{\beta}(\bsk)+B_{\alpha \beta}\mu_{1}(\bsk),
\end{align*}
with 
\begin{align*}
A_{\alpha \beta} = \frac{\alpha -1}{\beta -1} \quad \text{and} \quad B_{\alpha \beta} = \frac{\beta -\alpha}{\beta -1}.
\end{align*}
Here we note that $B_{\alpha \beta}>1/2$ since we impose the condition $\beta \geq 2\alpha$.

Using these facts, we have
\begin{align}\label{eq:main_part_proof}
 \sum_{\bsk\in P^{\perp}\setminus \{\bszero\}}b^{-2\mu_{\alpha}(\bsk)} & \leq \sum_{\bsk\in P^{\perp}\setminus \{\bszero\}}b^{-2A_{\alpha \beta}\mu_{\beta}(\bsk)-2B_{\alpha \beta}\mu_{1}(\bsk)} \nonumber \\
 & \leq b^{-2A_{\alpha \beta}\mu_{\beta}(P)}\sum_{\bsk\in P^{\perp}\setminus \{\bszero\}}b^{-2B_{\alpha \beta}\mu_{1}(\bsk)} \nonumber \\
 & = b^{-2A_{\alpha \beta}\mu_{\beta}(P)}\sum_{z=\mu_1(P)}^{\infty}b^{-2B_{\alpha \beta}z}\sum_{\substack{\bsk\in P^{\perp}\setminus \{\bszero\}\\ \mu_1(\bsk)=z}}1 \nonumber \\
 & = b^{-2A_{\alpha \beta}\mu_{\beta}(P)}\sum_{z=\mu_1(P)}^{\infty}b^{-2B_{\alpha \beta}z}\sum_{\substack{\bsl\in \NN_0^s\setminus \{\bszero\}\\ |\bsl|_1=z}}\sum_{\substack{\bsk\in P^{\perp}\setminus \{\bszero\}\\ \mu_1(k_j)=l_j, \forall j}}1 ,
\end{align}
where we denote $|\bsl|_1:=l_1+\cdots+l_s$. Since $P^{\perp}$ is a subgroup of $\{0,1,\ldots,b^{\beta gw}-1\}^s$ with the group operation $\oplus$, it follows from \cite[Lemma~2.2]{Skr06} that the innermost sum in the last expression can be bounded above by 
\begin{align*}
 \sum_{\substack{\bsk\in P^{\perp}\setminus \{\bszero\}\\ \mu_1(k_j)=l_j, \forall j}}1 \leq b^{|\bsl|_1-\mu_1(P)+1} = b^{z-\mu_1(P)+1},
\end{align*}
where the right-hand side is independent of the choice of $\bsl\in \NN_0^s$. Further the number of possible choices of $\bsl\in \NN_0^s$ with $|\bsl|_1=z$ is given by the usual binomial coefficient $\binom{z+s-1}{s-1}$.
Thus by using these results and the inequality
\begin{align*}
 \sum_{t=t_0}^{\infty}q^{-t}\binom {t+k-1}{k-1}\le q^{-t_0}\binom {t_0+k-1}{k-1}\left( 1-\frac{1}{q}\right)^{-k},
\end{align*}
which holds for any real number $q>1$ and any $k,t_0\in \NN$ (see for instance \cite[Lemma~13.24]{DPbook}) we have
\begin{align*}
 \sum_{z=\mu_1(P)}^{\infty}b^{-2B_{\alpha \beta}z}\sum_{\substack{\bsl\in \NN_0^s\setminus \{\bszero\}\\ |\bsl|_1=z}}\sum_{\substack{\bsk\in P^{\perp}\setminus \{\bszero\}\\ \mu_1(k_j)=l_j, \forall j}}1 
 & \leq \sum_{z=\mu_1(P)}^{\infty}b^{-2B_{\alpha \beta}z}\binom{z+s-1}{s-1}b^{z-\mu_1(P)+1} \\
 & = b^{-\mu_1(P)+1}\sum_{z=\mu_1(P)}^{\infty}b^{-(2B_{\alpha \beta}-1)z}\binom{z+s-1}{s-1} \\
 & \leq b^{-2B_{\alpha \beta}\mu_1(P)+1}\binom{\mu_1(P)+s-1}{s-1}\left( 1-b^{-(2B_{\alpha \beta}-1)}\right)^{-s} \\
 & \leq G_{\alpha,\beta,b,s}\frac{(\mu_1(P)+1)^{s-1}}{b^{2B_{\alpha \beta}\mu_1(P)}},
\end{align*}
with $G_{\alpha,\beta,b,s}=b\left( 1-b^{-(2B_{\alpha \beta}-1)}\right)^{-s}>0$, where the last inequality stems from the inequality
\begin{align*}
 \binom{\mu_1(P)+s-1}{s-1}=\prod_{j=1}^{s-1}\frac{\mu_1(P)+s-j}{s-j}\leq (\mu_1(P)+1)^{s-1}.
\end{align*}
Substituting the above bound into (\ref{eq:main_part_proof}) and using the the fact that $\mu_1(P)$ cannot be greater than $gw+1$, we have
\begin{align*}
 \sum_{\bsk\in P^{\perp}\setminus \{\bszero\}}b^{-2\mu_{\alpha}(\bsk)} & \leq G_{\alpha,\beta,b,s}\frac{(\mu_1(P)+1)^{s-1}}{b^{2A_{\alpha \beta}\mu_{\beta}(P)+2B_{\alpha \beta}\mu_1(P)}} \\
& \leq G_{\alpha,\beta,b,s}b^{2A_{\alpha \beta}t+2B_{\alpha \beta}t'}\frac{(gw+2)^{s-1}}{b^{2A_{\alpha \beta}\beta gw+2B_{\alpha \beta}gw}} \\
& \leq G'_{\alpha,\beta,b,s}\frac{(gw+2)^{s-1}}{b^{2\alpha gw}} \leq G''_{\alpha,\beta,b,s}\frac{(\log N)^{s-1}}{N^{2\alpha}}
\end{align*}
where $G'_{\alpha,\beta,b,s},G''_{\alpha,\beta,b,s}>0$ and we write $N=b^{gw}$. In summary, we have got an upper bound on the main part as follows.

\begin{proposition}\label{prop:main_part}
Let $s,\alpha\in \NN$, $\alpha \geq 2$.
Let $\beta,g\in \NN$ with $\beta \geq 2\alpha$, $g\geq 2\alpha s$ and $g\geq \lfloor s(\beta -1)/2\rfloor$, and let $b\geq \beta gs$ be a prime.
Then for $w\in \NN$, a digital net $P$ of cardinality $N=b^{gw}$ constructed as in Subsection~\ref{subsec:construction} satisfies
\begin{align*}
\sum_{\bsk,\bsl\in P^{\perp}\setminus \{\bszero\}}\left|\hat{\Kcal}_{\alpha,s}(\bsk,\bsl)\right| \leq A^{(1)}_{\alpha,\beta,b,s}\frac{(\log N)^{s-1}}{N^{2\alpha}},
\end{align*}
where $A^{(1)}_{\alpha,\beta,b,s}>0$.
\end{proposition}

%%%%%%%%%%%%%%%%%%%%%%%%%%%%%%%%%%%%%%%%%%%%%%%%%%%%%%%%%%%%%%%%%%%%%%%%%%%%%%%%%%%%%%%%%%%%%%%%%%%
\subsection{An upper bound on the discretization part}\label{subsec:disc_part}
Throughout this subsection, let $n=\beta gw$ for ease of notation.
Following an argument similar to that used in the main part, we have
\begin{align*}
\sum_{\substack{\bsk\in \NN_0^s \setminus \{0,1,\ldots,b^n-1\}^s \\ \bsl\in P_{\infty}^{\perp}\setminus \{\bszero\}}}\left|\hat{\Kcal}_{\alpha,s}(\bsk,\bsl)\right| & = \sum_{\substack{\bsk\in \NN_0^s \setminus \{0,1,\ldots,b^n-1\}^s \\ \bsl\in P_{\infty}^{\perp}\setminus \{\bszero\}}}\prod_{j=1}^{s}\left|\hat{\Kcal}_{\alpha}(k_j,l_j)\right| \\
& \leq \sum_{\substack{\bsk\in \NN_0^s \setminus \{0,1,\ldots,b^n-1\}^s \\ \bsl\in P_{\infty}^{\perp}\setminus \{\bszero\}}}\prod_{j\colon (k_j,l_j)\neq (0,0)}D_{\alpha,b}b^{-\mu_{\alpha}(k_j)-\mu_{\alpha}(l_j)}\\
& \leq \max\{1,D_{\alpha,b}^s\}\sum_{\substack{\bsk\in \NN_0^s \setminus \{0,1,\ldots,b^n-1\}^s \\ \bsl\in P_{\infty}^{\perp}\setminus \{\bszero\}}}b^{-\mu_{\alpha}(\bsk)-\mu_{\alpha}(\bsl)}\\
& = \max\{1,D_{\alpha,b}^s\}\sum_{\bsk\in \NN_0^s \setminus \{0,1,\ldots,b^n-1\}^s }b^{-\mu_{\alpha}(\bsk)}\sum_{\bsl\in P_{\infty}^{\perp}\setminus \{\bszero\}}b^{-\mu_{\alpha}(\bsl)},
\end{align*}
where we use Equation~(\ref{eq:walsh_coeff_projection}) and \cite[Proposition~20]{BD09} in the first equality and the inequality, respectively.

Since $P$ is an order $\beta$ digital $(t,gw,s)$-net over $\FF_b$ with $t=\beta \left\lfloor s(\beta-1)/2 \right\rfloor$ as stated in the last subsection, it follows from Lemma~\ref{lem:propagation} that $P$ is also an order $\alpha$ digital $(t',gw,s)$-net with
\begin{align*}
t' = \lceil t \alpha / \beta\rceil = \alpha \left\lfloor \frac{s(\beta-1)}{2} \right\rfloor 
\end{align*}
and the precision $n=\beta gw$. Applying the result of \cite[Lemma~15.20]{DPbook} with $t',\alpha/\beta$ (the notations used in this paper) substituted into $t,\beta$ therein, respectively, we have
\begin{align*}
\sum_{\bsl\in P_{\infty}^{\perp}\setminus \{\bszero\}}b^{-\mu_{\alpha}(\bsl)} & = \sum_{\emptyset \ne u\subseteq \{1,\ldots,s\}}\sum_{\substack{\bsl_u\in \NN^{|u|}\\ (\bsl_u,\bszero)\in P_{\infty}^{\perp}}}b^{-\mu_{\alpha}(\bsl_u)} \\
& \leq \sum_{\emptyset \ne u\subseteq \{1,\ldots,s\}}H_{\alpha,b,|u|}\frac{(n \alpha /\beta -t'+2)^{|u|\alpha}}{b^{n \alpha /\beta -t'}} \leq H'_{\alpha,\beta,b,s}\frac{n^{s\alpha}}{b^{n \alpha /\beta}}
\end{align*}
where $H_{\alpha,b,|u|}>0$ for all $\emptyset \ne u\subseteq \{1:s\}$ and $H'_{\alpha,\beta,b,s}>0$.

In what follows, let 
\begin{align*}
S_1 := \sum_{k=0}^{\infty}b^{-\mu_{\alpha}(k)}\quad \text{and} \quad S_{2,n} := \sum_{k=0}^{b^n-1}b^{-\mu_{\alpha}(k)}.
\end{align*}
Here we note that $S_1$ is finite when $\alpha \geq 2$, see for instance \cite[Lemma~15.33]{DPbook}.
Then we have
\begin{align*}
\sum_{\bsk\in \NN_0^s \setminus \{0,1,\ldots,b^n-1\}^s }b^{-\mu_{\alpha}(\bsk)} & = \sum_{\bsk\in \NN_0^s}b^{-\mu_{\alpha}(\bsk)} - \sum_{\bsk\in \{0,1,\ldots,b^n-1\}^s}b^{-\mu_{\alpha}(\bsk)} \\
& = \left( \sum_{k=0}^{\infty}b^{-\mu_{\alpha}(k)}\right)^s - \left( \sum_{k=0}^{b^n-1}b^{-\mu_{\alpha}(k)}\right)^s \\
& = S_1^s-S_{2,n}^s \\
& = (S_1 - S_{2,n})\sum_{a=0}^{s-1}S_1^a S_{2,n}^{s-1-a} \\
& \leq (S_1 - S_{2,n})\sum_{a=0}^{s-1}S_1^{s-1} = (S_1 - S_{2,n})sS_1^{s-1}.
\end{align*}
Therefore right now we have the following bound on the discretization part:
\begin{align}\label{eq:discre_part_bound}
\sum_{\substack{\bsk\in \NN_0^s \setminus \{0,1,\ldots,b^n-1\}^s \\ \bsl\in P_{\infty}^{\perp}\setminus \{\bszero\}}}\left|\hat{\Kcal}_{\alpha,s}(\bsk,\bsl)\right| \leq \frac{S_1 - S_{2,n}}{b^{n \alpha /\beta}}H''_{\alpha,\beta,b,s}n^{s\alpha},
\end{align}
where we set $H''_{\alpha,\beta,b,s}=sS_1^{s-1}H'_{\alpha,\beta,b,s}$.

In the following, we give a bound on $S_1-S_{2,n}$.
For an integer $k\geq b^n$, we denote its $b$-adic expansion by $k=\kappa_1 b^{a_1-1}+\cdots+\kappa_v b^{a_v-1}$ with $\kappa_1,\ldots,\kappa_v\in \{1,\ldots,b-1\}$, $a_1>\cdots > a_v>0$ and $a_1>n$.
We have
\begin{align*}
S_1 - S_{2,n} & = \sum_{k=b^n}^{\infty}b^{-\mu_{\alpha}(k)} \\
& = \sum_{v=1}^{\infty}\sum_{\kappa_1,\ldots,\kappa_v\in \{1,\ldots,b-1\}}\sum_{\substack{a_1>\cdots > a_v>0 \\ a_1>n}}b^{-\mu_{\alpha}(\kappa_1 b^{a_1-1}+\cdots+\kappa_v b^{a_v-1})} \\
& = \sum_{v=1}^{\infty}(b-1)^{v}\sum_{\substack{a_1>\cdots > a_v>0 \\ a_1>n}}b^{-\mu_{\alpha}(b^{a_1-1}+\cdots+ b^{a_v-1})} \\
& = \sum_{v=1}^{\alpha-1}(b-1)^{v}T_{v,n} + \sum_{v=\alpha}^{\infty}(b-1)^{v}U_{v,\alpha,n} ,
\end{align*}
where we write
\begin{align*}
T_{v,n} := \sum_{\substack{a_1>\cdots > a_v>0 \\ a_1>n}}b^{-(a_1+\cdots+a_v)} \quad \text{and}\quad U_{v,\alpha,n}:= \sum_{\substack{a_1>\cdots > a_v>0 \\ a_1>n}}b^{-(a_1+\cdots+a_{\alpha})} .
\end{align*}

For any $1\leq v< \alpha$ we have
\begin{align*}
T_{v,n} & = \sum_{a_v=1}^{\infty}\sum_{a_{v-1}=a_v+1}^{\infty}\cdots \sum_{a_2=a_3+1}^{\infty}\sum_{\substack{a_1=a_2+1\\ a_1>n}}^{\infty}b^{-(a_1+\cdots+a_v)} \\
& \leq  \sum_{a_v=1}^{\infty}b^{-a_v}\sum_{a_{v-1}=a_v+1}^{\infty}b^{-a_{v-1}}\cdots \sum_{a_2=a_3+1}^{\infty}b^{-a_2}\sum_{a_1=n+1}^{\infty}b^{-a_1} \\
& = \frac{1}{b^n(b-1)}\sum_{a_v=1}^{\infty}b^{-a_v}\sum_{a_{v-1}=a_v+1}^{\infty}b^{-a_{v-1}}\cdots \sum_{a_2=a_3+1}^{\infty}b^{-a_2} \\
& = \frac{1}{b^n(b-1)^2}\sum_{a_v=1}^{\infty}b^{-a_v}\sum_{a_{v-1}=a_v+1}^{\infty}b^{-a_{v-1}}\cdots \sum_{a_3=a_4+1}^{\infty}b^{-2a_3} \\
& \qquad \vdots \\
& = \frac{1}{b^n(b-1)}\prod_{i=1}^{v-1}\frac{1}{b^i-1},
\end{align*}
where the empty product equals 1. Similarly, for any $v\geq \alpha$ we have
\begin{align*}
U_{v,\alpha,n} & = \sum_{a_v=1}^{\infty}\sum_{a_{v-1}=a_v+1}^{\infty}\cdots \sum_{a_2=a_3+1}^{\infty}\sum_{\substack{a_1=a_2+1\\ a_1>n}}^{\infty}b^{-(a_1+\cdots+a_{\alpha})} \\
& \leq \sum_{a_v=1}^{\infty}\cdots \sum_{a_{\alpha +1}=a_{\alpha+2}+1}^{\infty}\sum_{a_{\alpha}=a_{\alpha+1}+1}^{\infty}b^{-a_{\alpha}}\cdots \sum_{a_2=a_3+1}^{\infty}b^{-a_2}\sum_{a_1=n+1}^{\infty}b^{-a_1} \\
& = \frac{1}{b^n(b-1)}\left( \prod_{i=1}^{\alpha-1}\frac{1}{b^i-1}\right) \sum_{a_v=1}^{\infty}\cdots \sum_{a_{\alpha +1}=a_{\alpha+2}+1}^{\infty}b^{-(\alpha-1) a_{\alpha+1}} \\
& = \frac{1}{b^n(b-1)}\left( \prod_{i=1}^{\alpha-1}\frac{1}{b^i-1}\right) \left( \frac{1}{b^{\alpha-1}-1}\right)^{v-\alpha}.
\end{align*}
Therefore, $S_1-S_{2,n}$ can be bounded above by
\begin{align}\label{eq:bound_discre}
S_1 - S_{2,n} & \leq \sum_{v=1}^{\alpha-1}\frac{1}{b^n}\prod_{i=1}^{v-1}\frac{b-1}{b^i-1} + \sum_{v=\alpha}^{\infty}\frac{1}{b^n}\left( \prod_{i=1}^{\alpha-1}\frac{b-1}{b^i-1}\right) \left( \frac{b-1}{b^{\alpha-1}-1}\right)^{v-\alpha} \\
& = \frac{1}{b^n}\left[ \sum_{v=1}^{\alpha-1}\prod_{i=1}^{v-1}\frac{b-1}{b^i-1} + \frac{b^{\alpha-1}-1}{b^{\alpha-1}-b}\prod_{i=1}^{\alpha-1}\frac{b-1}{b^i-1} \right] =: \frac{B_{\alpha,b}}{b^n}, \nonumber
\end{align}
when $\alpha\geq 3$. Note that the second term of (\ref{eq:bound_discre}) does not converge when $\alpha=2$, and thus, we need a further argument to obtain a bound on $S_1-S_{2,n}$ as below.

Let $\alpha=2$. For any $v> n$ it obviously holds that $a_1>n$, so that we have
\begin{align*}
U_{v,2,n} & = \sum_{a_v=1}^{\infty}\sum_{a_{v-1}=a_v+1}^{\infty}\cdots \sum_{a_2=a_3+1}^{\infty}\sum_{a_1=a_2+1}^{\infty}b^{-(a_1+a_2)} \\
& = \frac{1}{b-1}\left( \frac{1}{b^2-1}\right)^{v-1}
\end{align*}
Applying this bound on $U_{v,2,n}$, $S_1-S_{2,n}$ can be bounded above by
\begin{align*}
S_1 - S_{2,n} & = (b-1)T_{1,n} + \sum_{v=2}^{n}(b-1)^{v}U_{v,2,n}+\sum_{v=n+1}^{\infty}(b-1)^{v}U_{v,2,n} \\
& \leq \frac{1}{b^n} + \sum_{v=2}^{n}\frac{1}{b^n}+ \sum_{v=n+1}^{\infty}\left( \frac{1}{b+1}\right)^{v-1} \\
& = \frac{n}{b^n} + \frac{1}{b}\left( \frac{1}{b+1}\right)^{n-1} \leq \frac{n+1}{b^n}\le \frac{2n}{b^n}.
\end{align*}

Applying the above bounds on $S_1-S_{2,n}$ to the right-hand side of (\ref{eq:discre_part_bound}), it follows that
\begin{align*}
\sum_{\substack{\bsk\in \NN_0^s \setminus \{0,1,\ldots,b^n-1\}^s \\ \bsl\in P_{\infty}^{\perp}\setminus \{\bszero\}}}\left|\hat{\Kcal}_{\alpha,s}(\bsk,\bsl)\right| \leq \frac{1}{b^{n (\alpha /\beta+1)}}B_{\alpha,b}H''_{\alpha,\beta,b,s}n^{s\alpha}.
\end{align*}
for $\alpha\geq 3$, and
\begin{align*}
\sum_{\substack{\bsk\in \NN_0^s \setminus \{0,1,\ldots,b^n-1\}^s \\ \bsl\in P_{\infty}^{\perp}\setminus \{\bszero\}}}\left|\hat{\Kcal}_{\alpha,s}(\bsk,\bsl)\right| \leq \frac{1}{b^{n (\alpha /\beta+1)}}2 H''_{\alpha,\beta,b,s}n^{s\alpha +1}.
\end{align*}
for $\alpha=2$. Let us recall that we used the notation $n=\beta gw$ and that $P$ consists of $N=b^{gw}$ points and $\beta \geq 2\alpha$. In summary, we have got an upper bound on the discretization part as follows.

\begin{proposition}\label{prop:discre_part}
Let $s,\alpha\in \NN$, $\alpha \geq 2$.
Let $\beta,g\in \NN$ with $\beta \geq 2\alpha$, $g\geq 2\alpha s$ and $g\geq \lfloor s(\beta -1)/2\rfloor$, and let $b\geq \beta gs$ be a prime.
Then for $w\in \NN$, the discretization part is bounded above by
\begin{align*}
\sum_{\substack{\bsk\in \NN_0^s \setminus \{0,1,\ldots,b^{\beta gw}-1\}^s \\ \bsl\in P_{\infty}^{\perp}\setminus \{\bszero\}}}\left|\hat{\Kcal}_{\alpha,s}(\bsk,\bsl)\right| \leq A^{(2)}_{\alpha,\beta,b,s}\frac{(\log N)^{s\alpha}}{N^{3\alpha}} ,
\end{align*}
when $\alpha\geq 3$, where $A^{(2)}_{\alpha,\beta,b,s}>0$. Similarly we have 
\begin{align*}
\sum_{\substack{\bsk\in \NN_0^s \setminus \{0,1,\ldots,b^{\beta gw}-1\}^s \\ \bsl\in P_{\infty}^{\perp}\setminus \{\bszero\}}}\left|\hat{\Kcal}_{\alpha,s}(\bsk,\bsl)\right| \leq A^{(3)}_{\beta,b,s}\frac{(\log N)^{s\alpha+1}}{N^{3\alpha}} ,
\end{align*}
when $\alpha=2$, where $A^{(3)}_{\beta,b,s}>0$.
\end{proposition}

%%%%%%%%%%%%%%%%%%%%%%%%%%%%%%%%%%%%%%%%%%%%%%%%%%%%%%%%%%%
%%%%%%%%%%%%%%%%%%%%%%%%%%%%%%%%%%%%%%%%%%%%%%%%%%%%%%%%%%%
%%%%%%%%%%%%%%%%%%%%%%%%%%%%%%%%%%%%%%%%%%%%%%%%%%%%%%%%%%%


\begin{thebibliography}{99}
\bibitem{Aro50} N. Aronszajn, Theory of reproducing kernels, Trans. Amer. Math. Soc. 68 (1950), 337--404.
\bibitem{BD09} J. Baldeaux and J. Dick, QMC rules of arbitrary high order: Reproducing kernel Hilbert space approach, Constr. Approx. 30 (2009), 495--527.
\bibitem{BDP11} J. Baldeaux, J. Dick and F. Pillichshammer, Duality theory and propagation rules for higher order nets, Discrete Math. 311 (2011), 362--386.
\bibitem{CS02} W.~W.~L. Chen and M.~M. Skriganov, Explicit constructions in the classical mean squares problem in irregularities of point distribution, J. Reine Angew. Math. 545 (2002), 67--95.
\bibitem{Dick07} J. Dick, Explicit constructions of quasi-Monte Carlo rules for the numerical integration of high-dimensional periodic functions, SIAM J. Numer. Anal. 45 (2007), 2141--2176.
\bibitem{Dick08} J. Dick, Walsh spaces containing smooth functions and quasi-Monte Carlo rules of arbitrary high order, SIAM J. Numer. Anal. 46 (2008), 1519--1553.
\bibitem{Dick09} J. Dick, The decay of the Walsh coefficients of smooth functions, Bull. Aust. Math. Soc. 80 (2009), 430--453.
\bibitem{DKLNS14} J. Dick, F.~Y. Kuo, Q.~T. Le Qia, D. Nuyens and C. Schwab, Higher order QMC Petrov--Galerkin discretization for affine parametric operator equations with random field inputs, SIAM J. Numer. Anal. 52 (2014), 2676--2702.
\bibitem{DPbook} J. Dick and F. Pillichshammer, \emph{Digital Nets and Sequences: Discrepancy Theory and Quasi-Monte Carlo Integration}, Cambridge University Press, Cambridge, 2010.
\bibitem{Faure82} H. Faure, Discr\'{e}pances de suites associ\'{e}es \`{a} un syst\`{e}me de num\'{e}ration (en dimension s), Acta Arith. 41 (1982), 337--351.
\bibitem{Frolov76} K.~K. Frolov, Upper error bounds for quadrature formulas on function classes, Dokl. Akad. Nauk SSSR 231 (1976), 818--821.
\bibitem{GSYtent1} T. Goda, K. Suzuki and T. Yoshiki, The $b$-adic tent transformation for quasi-Monte Carlo integration using digital nets, J. Approx. Theory 194 (2015), 62--86.
\bibitem{GSYtent2} T. Goda, K. Suzuki and T. Yoshiki, Digital nets with infinite digit expansions and construction of folded digital nets for quasi-Monte Carlo integration, J. Complexity (2015) DOI:10.1016/j.jco.2015.09.005.
\bibitem{GSYexist} T. Goda, K. Suzuki and T. Yoshiki, Optimal order quasi-Monte Carlo integration in weighted Sobolev spaces of arbitrary smoothness, ArXiv Preprint arXiv:1508.06373.
\bibitem{Hick98} F.~J. Hickernell, A generalized discrepancy and quadrature error bound, Math. Comp. 67 (1998), 299--322.
\bibitem{HMOT15} A. Hinrichs, L. Markhasin, J. Oettershagen and T. Ullrich, Optimal quasi-Monte Carlo rules on order 2 digital nets for the numerical integration of multivariate periodic functions, Numer. Math. (2015) DOI:10.1007/s00211-015-0765-y.
\bibitem{KSS12} F.~Y. Kuo, C. Schwab and I.~H. Sloan, Quasi-Monte Carlo finite element methods for a class of elliptic partial differential equations with random coefficients, SIAM J. Numer. Anal. {\bf 50} (2012), 3351--3374.
\bibitem{Nied86} H. Niederreiter, Low-discrepancy point sets, Monatsh. Math. 102 (1986), 155--167.
\bibitem{Nied88} H. Niederreiter, Low-discrepancy and low-dispersion sequences, J. Number Theory 30 (1988), 51--70.
\bibitem{Niebook} H. Niederreiter, \emph{Random Number Generation and Quasi-Monte Carlo Methods}, CBMS-NSF Regional Conference Series in Applied Mathematics 63, SIAM, Philadelphia, 1992.
\bibitem{NXbook} H. Niederreiter and C.~P. Xing, \emph{Rational Points on Curves over Finite Fields: Theory and Applications}, London Mathematical Society Lecture Note Series 285, Cambridge University Press, Cambridge, 2001.
\bibitem{NWbook} E. Novak and H. Wo\'zniakowski, \emph{Tractability of Multivariate Problems, Volume I: Linear Information}, EMC Tracts in Mathematics 6, European Mathematical Society, Z\"urich, 2008.
\bibitem{RT97} M.~Yu. Rosenbloom and M.~A. Tsfasman, Codes for the $m$-metric, Probl. Inf. Transm. 33 (1997), 55--63.
\bibitem{Skr06} M.~M. Skriganov, Harmonic analysis on totally disconnected groups and irregularities of point distributions, J. Reine Angew. Math. 600 (2006), 25--49.
\bibitem{SJbook} I.~H. Sloan and S. Joe, \emph{Lattice Methods for Multiple Integration}, Oxford Science Publications, New York, 1994.
\bibitem{SW98} I.~H. Sloan and H. Wo\'zniakowski, When are quasi-Monte Carlo algorithms efficient for high-dimensional integrals?, J. Complexity 14 (1998), 1--33.
\bibitem{Sobol67} I.~M. Sobol', The distribution of points in a cube and approximate evaluation of integrals, Zh. Vycisl. Mat. i Mat. Fiz. 7 (1967), 784--802.
\bibitem{SYwalsh} K. Suzuki and T. Yoshiki, Formulas for the Walsh coefficients of smooth functions and their application to bounds on the Walsh coefficients, J. Approx. Theory (2016), DOI:10.1016/j.jat.2015.12.002.
\bibitem{Ullrich14} M. Ullrich, On ``Upper error bounds for quadrature formulas on function classes'' by K.~K. Frolov, ArXiv Preprint arXiv:1404.5457.
\bibitem{UU15} M. Ullrich and T. Ullrich, The role of Frolov's cubature formula for functions with bounded mixed derivative, ArXiv Preprint arXiv:1503.08846.
\bibitem{Wbook} G. Wahba, \emph{Spline Models for Observational Data}, CBMS-NSF Regional Conference Series in Applied Mathematics 59, SIAM, Philadelphia, 1990.
\bibitem{Ywalsh} T. Yoshiki, Bounds on Walsh coefficients by dyadic difference and a new Koksma-Hlawka type inequality for Quasi-Monte Carlo integration, ArXiv Preprint arXiv:1504.03175.
\end{thebibliography}
\end{document}